\documentclass{article}

\usepackage{amsmath,amsxtra,amscd,amssymb,latexsym,stmaryrd,amsthm}
\usepackage{graphicx,color}
\usepackage{hyperref}
\usepackage{mathrsfs}
\usepackage[all]{xy}
\usepackage{tikz,xcolor}

\title{Thermodynamic Formalism for
Iterated Function Systems with Weights}

\author{
  L. Cioletti
  \\[-0.1cm]
  \footnotesize Departamento de Matem\'atica - UnB
  \\[-0.1cm]
  \footnotesize 70910-900, Bras\'ilia, Brazil
  \\[-0.1cm]
  \footnotesize\texttt{cioletti@mat.unb.br}
  \and
  \large
  Elismar R. Oliveira
  \\[-0.1cm]
  \footnotesize Departamento de Matem\'atica - UFRGS
  \\[-0.1cm]
  \footnotesize 91509-900, Porto Alegre, Brazil
  \\[-0.1cm]
  \footnotesize\texttt{oliveira.elismar@gmail.com}
}

\date{\small\today}

\newtheorem{theorem}{Theorem}

\newtheorem{corollary}{Corollary}
\newtheorem{lemma}{Lemma}
\newtheorem{proposition}{Proposition}

\theoremstyle{definition}
\newtheorem{definition}[theorem]{Definition}
\newtheorem{remark}[theorem]{Remark}
\newtheorem{example}[theorem]{Example}

\begin{document}
    \makeatletter
    \def\blfootnote{\gdef\@thefnmark{}\@footnotetext}
    \let\@fnsymbol\@roman
    \makeatother

\maketitle

\begin{abstract}
This paper introduces an intrinsic theory of
Thermodynamic Formalism for Iterated Functions Systems
with general positive continuous weights (IFSw).
We study the spectral properties of the
Transfer and Markov operators and one of our first
results is the proof of the existence of at least one eigenprobability
for the Markov operator associated to a positive eigenvalue.
Sufficient conditions are provided for this eingenvalue to be
the spectral radius of the transfer operator and we also prove in this
general setting that positive eigenfunctions of
the transfer operator are always associated to its spectral radius.

We introduce variational formulations for the topological entropy
of holonomic measures and the topological pressure of IFSw's
with weights given by a potential.
A definition of equilibrium state is then natural and we prove
its existence for any continuous potential.
We show, in this setting, a uniqueness result for the equilibrium state
requiring only the G\^ateaux differentiability
of the pressure functional.
We  also recover the classical formula relating the  powers of the transfer operator and the
topological pressure and establish its uniform convergence. In the last section we present some
examples and show that the results obtained can be viewed as
a generalization of several classical results in Thermodynamic Formalism
for ordinary dynamical systems.
\end{abstract}

\blfootnote{\textup{2010} \textit{Mathematics Subject Classification}: 37D35, 37B40, 28Dxx.}

\section{Introduction}

Thermodynamic Formalism has its roots in the seminal paper
by David Ruelle \cite{MR0234697}, where the transfer operator was introduced.
In that work the transfer operator was used to prove the uniqueness of the
Gibbs measures for some long-range Statistical Mechanics models in the
one-dimensional lattice.
In \cite{MR0399421},  Yakov Sinai established
a deep connection between one-dimensional Statistical Mechanics
and Hyperbolic Dynamical Systems on compact manifolds by using what
he called ``Markov Partition''.
Since then the Thermodynamic Formalism has attracted a
lot of attention in the Dynamical Systems
community because of its numerous important applications.
The transfer operator is a major tool in this theory and
has remarkable applications to topological dynamics, meromorphy
of dynamical zeta functions and multifractal analysis, just to name a few,
see \cite{MR3024807,MR2423393,MR556580,MR1143171,MR742227,MR1085356,MR1920859,MR0399421}
and references therein.

In the last decades, an interest in
bringing the techniques developed in Thermodynamic Formalism
for classical dynamical systems to the theory of Iterated Function Systems (IFS) has surfaced, but although Thermodynamic Formalism in the
context of IFS has been discussed by many authors the set of results obtained
thus far lack an uniform approach.
The best effort has been exerted via
Multifractal analysis and Ergodic optimization,
however, in the majority of the papers surrounding this matter,
hypothesis on the IFS self-maps and the weights such as
contractiveness, open set condition (OSC), non-overlaping,
conformality and  H\"older or Lipschtz regularity are required.

A first version of the
Ruele-Perron-Frobenius theorem for contractive IFS, via shift conjugation, was
introduced in \cite{MR1669203} --- the authors study the Hausdorff dimension
of the Gibbs measure under conformal and OSC conditions,
but several other important tools of Thermodynamic Formalism
were not considered.
Their proofs rely strongly on  contractiveness to
built a conjugation between the attractor and a code space, which is
the usual shift, bringing back the classic results of the
Ruelle-Perron–Frobenius theory.

Subsequently, \cite{MR1786722} considered the Topological Pressure,
Perron–Frobenius type operators, conformal measures and the
Hausdorff dimension of the limit set, making use of hypothesis such
as conformality, OCS and bounded distortion property (BDP).

Two years later, in \cite{MR1912488}, a developed model of
Thermodynamic Formalism arises and results for 
Conformal IFS and on Multifractal Analysis were obtained. 
The authors consider $X \subset \mathbb{R}^{d}$, a countable index set 
$I$ and a collection of uniform contractive injections $\tau_i: X \to X$.
They define a conformal infinite IFS $(X, (\tau_i)_{i \in I})$ and  
study it by using a coding map $\pi: I^{\infty} \to X$, 
given by 
$\displaystyle \pi(w)=\cap_{n\geq 0} \tau_{w_{0}}\cdots\tau_{w_{n-1}}(X)$. 
The contraction hypotheses on that work
are paramount to identify
the state space $X$ with a suitable symbolic space. 
As usual, OSC, conformality, BDP and other properties are assumed.
Their main tool is the topological pressure, defined as
$P(t)=\lim_{n \to \infty} \frac{1}{n} \, \ln \sum \| \cdot\|^t, $
where the sum is taken over the words of length
$n$ in $I^{\infty}$ and $\| \cdot\| $ is the norm of the 
derivative of the IFS maps. 
Under these conditions the remarkable formula
$HD(J\equiv \pi(I^{\infty}))=\inf\{t\geq 0 | P(t)\leq 0\}$
is obtained. They also deduce a variational formulation for 
the topological pressure for a potential,  which is
\[P(F)=\sup_{\mu} \left\{ h_{\mu} (\sigma) + \int f(w) d\mu(w) \right\},\]
where $h_{\mu} (\sigma)$ is the usual entropy for $\mu$ as a
shift invariant measure and $f$ is an amalgamated 
potential in $I^{\infty}$: $f(w)=f_{w_0}(\pi(\sigma(w))$
induced by $F=\{f_i: X \to \mathbb{C}\}_{i\in I}$ 
a family H\"older exponential weight functions. 
The generalization of the classical formula for the pressure is obtained in this 
context, i.e.
\[P(F)=\inf_{n \geq 1} \left\{ \frac{1}{n} \, \ln \sum_{|w|=n} \| Z_{n}(F)\| \right\},\]
where $Z_{n}(F)$ is suitable partition function (see \cite{MR1912488} for details). 
From this, equilibrium states are studied.

Some of the main results of this paper are inspired by three recent works
\cite{MR3377291,MR2097242,CER-2016}.
In \cite{MR3377291} Thermodynamic Formalism is developed for
H\"older potentials with the dynamics given by the left-shift mapping
acting on an infinite cartesian product of compact metric spaces,
which can be seen as an infinite contractive IFS.
In \cite{MR2097242} the author developed a
Thermodynamical Formalism for finite iterated function system, including
self-affine and self-conformal maps. The topological pressure is defined
for cylinder functions, but the measures are defined on an ad-hoc
symbolic space and not on the state space itself,
and contractiveness is assumed on the IFS to project it onto the code space.
A strong control of overlapping cannot be dismissed.
One of our main results Theorem \ref{prop-exist-auto-medida}
is inspired by an analogous result in \cite{CER-2016}, for symbolic
dynamics on uncoutable alphabets.

In order to construct  a truly intrinsic
theory of Thermodynamic Formalism for IFS, conditions such as
contractiveness or conformality should be avoided
and, ideally, only the continuity of the self-maps ought to be requested.
To have a theory where phase transition phenomenon can manifest,
continuity (or a less restrictive condition) on the weights
is a very natural condition to be considered.

Aiming to have an intrinsic theory, we extended
some of the results obtained in \cite{MR2461833}.
The main achievement of that paper was to
introduce the idea of holonomic probability measure as the
natural replacement of invariance for an IFS.
The authors considered a variational notion of entropy and topological pressure
in terms of the transfer operator,
leading to the existence of a very natural variational principle.
However, abundant information on the spectral properties of the
transfer operator was required and strong hypothesis on the IFS
were needed. Here we show how to avoid such requirements
for finite continuous IFS with positive continuous weights, and
 remark that no coding spaces are used in our work and
only intrinsic structures of the IFS are necessary to
define and study equilibrium states as probability measures
on the state space.

This work adds to the remarkable
effort that has been done by several authors in the last few years
to produce more general results for IFS, regarding fundamental results such as the
existence of attractors and attractive measures for the Markov operator.
See the recent works \cite{MR3623766} for weakly hyperbolic IFS, \cite{2016arXiv160706165M}
for $P$-weakly hyperbolic IFS and \cite{2016arXiv160502752M} for non-hyperbolic IFS.

The paper is organized as follows.
In Section~\ref{IFS operators} we introduce the notion of IFS
with weights, generalizing the basic setting where the idea of potential
is present. We also study the associated transfer operator and present a
characterization of its spectral radius.
In Section \ref{markov op proper}, we study the
Markov operator from an abstract point of view and prove a very
general result regarding the existence of eigenmeasures
associated to a positive eigenvalue.
The holonomic measures are introduced in
Section~\ref{holon measures} and their existence,
disintegration and the compatibility with IFS are carefully discussed.
In the subsequent section, the notions of entropy for a holonomic probability
measure and topological pressure for a continuous potential are introduced and
we show that, for any continuous potential and IFS, there is always an equilibrium state.
We also prove that the G\^ateaux differentiability of the topological pressure
is linked to the uniqueness problem of equilibrium states.
By imposing extra conditions
on the IFS and the weights, we provide a constructive approach
to the existence of an equilibrium state by using the maximal eigendata
of the transfer operator in Section~\ref{applic eq state}.

\section{IFS with weights - Transfer and Markov\break  Operators}\label{IFS operators}

In this section we set up the basic notation and
present a fundamental result about the eigenspace associated to
the maximal eigenvalue (or spectral radius) of transfer operator.

\bigskip

Throughout this paper, $X$, $Y$ and $\Omega$ are general compact metric spaces.
The Borel $\sigma$-algebra of $X$ is denoted by $\mathscr{B}(X)$
and similar notation is used for $Y$ and $\Omega$. The Banach
space of all real continuous functions equipped with supremum norm
is denoted by $C(X,\mathbb{R})$.
Its topological dual, as usual, is identified with $\mathscr{M}_{s}(X)$,
the space of all finite Borel signed measures endowed with
total variation norm. We use the notation $\mathscr{M}_{1}(X)$
for the set of all Borel probability measures over $X$
supplied with the weak-$*$ topology. Since we are assuming
that $X$ is compact metric space then we have that the topological
space $\mathscr{M}_{1}(X)$ is compact and metrizable.

An Iterated Function System with weights, IFSw for short,  is an ordered  triple
$\mathcal{R}_{  q}=(X,   \tau,   q)$, where
$  \tau \equiv \{\tau_0, ..., \tau_{n-1}\}$
is a collection of $n$ continuous functions from $X$ to itself and
$  q \equiv \{q_0,..., q_{n-1}\} $ is a finite collection of continuous
weights from $X$ to $[0,\infty)$.
An IFSw $\mathcal{R}_{  q}$ is said to be {\bf normalized} if for all $x\in X$
we have $\sum_{i=0}^{n-1} q_i(x)=1$.
A normalized IFSw with nonnegative weight is sometimes called
an IFS with place dependent probabilities (IFSpdp).

\begin{definition} Let $\mathcal{R}_{  q}=(X,   \tau,   q)$ be an IFSw.
The Transfer and Markov operators associated to $\mathcal{R}_{  q}$ are defined as follows:
\begin{enumerate}
  \item
  The \textbf{Transfer Operator}
  $B_{  q}: C(X , \mathbb{R}) \to C(X, \mathbb{R})$
  is given by
  \[
  B_{  q}(f) (x)= \sum_{i=0}^{n-1} q_{i}(x) f(\tau_{j}(x)),
  \qquad \forall x \in X.
  \]

  \item
  The \textbf{Markov Operator}
  $\mathcal{L}_{  q} : \mathscr{M}_{s}(X) \to \mathscr{M}_s(X)$
  is the unique bounded linear operator satisfying
  \[
  \int_{X} f\, d[\mathcal{L}_{  q} (\mu)]
  =
  \int_{X } B_{  q}(f)\,  d\mu,
  \]
  for all $\mu \in \mathscr{M}_s(X)$ and $f\in C(X,\mathbb{R})$.
\end{enumerate}
\end{definition}

Note that the isomorphism $C(X,\mathbb{R})^*\backsimeq \mathscr{M}_{s}(X)$
allow us to look at the Markov operator $\mathcal{L}_{  q}$
as the Banach transpose of $B_{  q}$, i.e., $B_{  q}^*=\mathcal{L}_{  q}$.

\bigskip

\begin{proposition} \label{prop-powers-1}
	Let $\mathcal{R}_{  q}=(X,   \tau,   q)$ be a continuous IFSw.
	Then for the $N$-th iteration of $B_q$ we have
	\[
	B_{  q}^N(1) (x)
	=
	\sum_{w_{0} \ldots w_{N-1} =0}^{N-1}
	P_{x}^{\hspace*{0.03cm}  q}(w_{0},\ldots, w_{N-1})
	\]
	where,
	$P_{x}^{\hspace*{0.03cm}  q}(w_0,\ldots, w_{N-1})
	\equiv \prod_{j=0}^{N-1}q_{w_j}(x_j)$,
	$x_0=x$ and $x_{j+1}=\tau_{w_j}x_j$.
\end{proposition}
\begin{proof}
This expression can be obtained by proceeding a formal induction on $N$.
\end{proof}

In the sequel we prove the main result of this section
which is Lemma \ref{powers description}.
Roughly speaking it states that any possible positive eigenfunctions
of the transfer operator $B_{ {q}}$ lives in the eigenspace associated to
the spectral radius of this operator acting on $C(X,\mathbb{R})$.
Beyond this nice application the lemma will be used, in the last section,
to derive an expression for the topological pressure
(see Definition \ref{Pressure}) involving only
the transfer operator and its powers.

\begin{lemma} \label{powers description}
Let $\mathcal{R}_{  q}=(X,   \tau,   q)$ be a continuous IFSw
and suppose that there are a positive number $\rho$ and
a strictly positive continuous function
$h: X \to \mathbb{R}$ such that  $B_{  q}(h)=\rho h$.
Then the following limit exits
\begin{align}\label{1sobreNlnBN}
\lim_{N \to \infty}
\frac{1}{N}
\ln\left(B_{  q}^N(1) (x) \right)
=
\ln \rho( B_{ {q}} )
\end{align}
the convergence is uniform in $x$ and
$\rho=\rho(B_{  q})$, the spectral radius of $B_{  q}$ acting on $C(X,\mathbb{R})$.
\end{lemma}
\begin{proof}
From the hypothesis we get a normalized continuous
IFSpdp $\mathcal{R}_{  p}=(X,   \tau,   p)$,
where the weights are given by
\[
p_j (x)
=
q_j(x)\frac{ h(\tau_j(x))}{\rho \,h(x)} , \; j=0, ..., n-1.
\]
Note that
$P_{x}^{\hspace*{0.03cm}  q}(w_{0}, \ldots, w_{N-1})$ and
$P_{x}^{\hspace*{0.03cm}  p}(w_{0}, \ldots, w_{N-1})$
are related in the following way
\begin{align*}
  P_{x}^{\hspace*{0.03cm}  q}(w_{0},\ldots, w_{N-1})
  &=
  \prod_{j=0}^{N-1} q_{w_j}(x_j)
  =
  \prod_{j=0}^{N-1} p_{w_j}(x_j) \frac{\rho \,h(x_j)}{h(x_{j+1})}
  \\
  &=
  \rho^N \frac{\,h(x_0)}{h(x_{N})}\prod_{j=0}^{N-1} p_{w_j}(x_j)
  \\
  &=
  P_{x}^{\hspace*{0.03cm}  p}(w_{0}, \ldots, w_{N-1}) \rho^N \frac{h(x_0)}{h(x_{N})}.
\end{align*}
Since $X$ is compact and $h$ is strictly positive and continuous, we have
for some positive constants $a$ and $b$ the following inequalities
$0< a\leq h(x_0)/h(x_{N})\leq b$.

By using the Proposition \ref{prop-powers-1} and the above equality,
we get for any fixed $N\in\mathbb{N}$ the following expression
\begin{align*}
  \frac{1}{N} \ln\left(B_{  q}^N(1) (x)) \right)
  & = \frac{1}{N} \ln\left( \sum_{w_{0} \ldots w_{N-1} =0}^{n-1} P_{x}^{\hspace*{0.03cm}  q}(w_{0}, \ldots, w_{N-1}) \right) \\
  & = \frac{1}{N} \ln\left( \sum_{w_{0} \ldots w_{N-1} =0}^{n-1} P_{x}^{\hspace*{0.03cm}  p}(w_{0}, \ldots, w_{N-1}) \rho^N \frac{h(x_0)}{h(x_{N})} \right) \\
  & = \ln \rho  + O(1/N)+ \frac{1}{N} \ln\left( \sum_{w_{0} \ldots w_{N-1} =0}^{n-1} P_{x}^{\hspace*{0.03cm}  p}(w_{0}, \ldots, w_{N-1}) \right)\\
  & = \ln \rho  + O(1/N),
\end{align*}
where the term $O(1/n)$ is independent of $x$. 
Therefore for every $N\geq 1$  we have
\[
\sup_{x\in X}
\left|   \frac{1}{N} \ln\left(B_{  q}^N(1) (x)) \right) - \ln \rho \right|
=
O(1/N).
\]
which proves \eqref{1sobreNlnBN}. From the above inequality and Gelfand's formula
for the spectral radius we have
\begin{align*}
\left|   \ln \rho(B_{  q}) - \ln \rho \right|
&=
\left|   \ln \left( \lim_{N\to\infty} \|B^N_{  q}\|^{\frac{1}{N}}\right) - \ln \rho \right|
=
\lim_{N\to\infty}
\left| \frac{1}{N} \ln \|B^N_{  q}\| - \ln \rho \right|
\\
&\leq
\limsup_{N\to\infty}
\ \ \sup_{x\in X}
\left|   \frac{1}{N} \ln\left(B_{  q}^N(1) (x)) \right) - \ln \rho \right|
\\
&\leq
\limsup_{N\to\infty}
\frac{C}{N}
=0.
\qedhere
\end{align*}
\end{proof}

\begin{remark}
In the special case, where there is a positive continuous function $\psi:X\to\mathbb{R}$
(potential) such that $q_i(x)\equiv\psi(\tau_i (x))$ for all $i=0,1,\ldots,n-1$ and $x\in X$
we will show later that $\ln \rho(B_{ {q}})$
is the topological pressure of $\psi$ (Definition \ref{Pressure} ).
\end{remark}

In the following example, Lemma \ref{powers description} is used to show that
lack of contractiveness and existence of a continuous eigenfunction associated 
to a positive eigenvalue, in general, impose very strong restrictions on the weights. 

\begin{example}\label{bad example}
We take $X=[0,1]$ with its usual topology and $\tau_i(x):=(-1)^i x + i$, for $i=0,1$. 
For any choice of a continuous weight $q$, we have  
that $\mathcal{R}_{  q}=(X,\tau,q)$ is a continuous IFSw. 
Since the derivative of $\tau_i(x)$ is equals to $(-1)^{i}$, 
for any point $x\in (0,1)$, no hyperbolicity exists in this system. 
Moreover, for each $t \in [0,1/2]$ the set $A_t=[t, 1-t]$ 
is fixed by the Hutchinson-Barnsley operator 
$A\longmapsto F_{R}(A)\equiv\tau_0(A)\cup\tau_1(A)$,
so this IFS has no attractor. 
Note that  $\cap_{n} \tau_{w_{0}} \circ\cdots \circ\tau_{w_{n-1}}([0,1])$  
is \textbf{never} a singleton, so this example do not fit
the recent theory of weakly hyperbolic sequences developed 
in \cite{2016arXiv160502752M,2016arXiv160706165M} 
neither \cite{MR3623766}. 

For $i=0,1$ consider the weight $q_i(x)=\psi(\tau_i(x))$, where 
$\psi:X\to\mathbb{R}$ a continuous and strictly positive function. 
In this case, the transfer operator 
is given by $B_{  q}(1)(x)=\psi(x) + \psi(1-x)$ and 
a simple induction on $N$, shows that $B_{q}^N(1)(x)=(\psi(x) +\psi(1-x))^N$
for all $N\geq 1$. Therefore
\[ 
\lim_{N \to \infty}
\frac{1}{N}
\ln\left(B_{  q}^N(1) (x) \right)
= \lim_{N \to \infty}
\frac{1}{N}
\ln\left(B_{  q}(1) (x) \right)^N= \ln (\psi(x) +\psi(1-x)).
\]
Because of Lemma \ref{powers description}, 
unless $\psi$ is chosen such that $\psi(x) +\psi(1-x)= \rho(B_{q}), \forall x \in X$,
no positive continuous eigenfunction for $B_q$, associated to a positive 
eigenvalue, can exist. 

On the other hand, if $\psi(x) +\psi(1-x)\equiv c>0, \forall x \in X$, 
then $B_{q}(1)= c \cdot 1$, i.e., the constant function $h\equiv 1$ 
is a continuous positive eigenfunction and $c=\rho(B_q)$.  
\begin{figure}[h]
\centering
\includegraphics[scale=0.7]{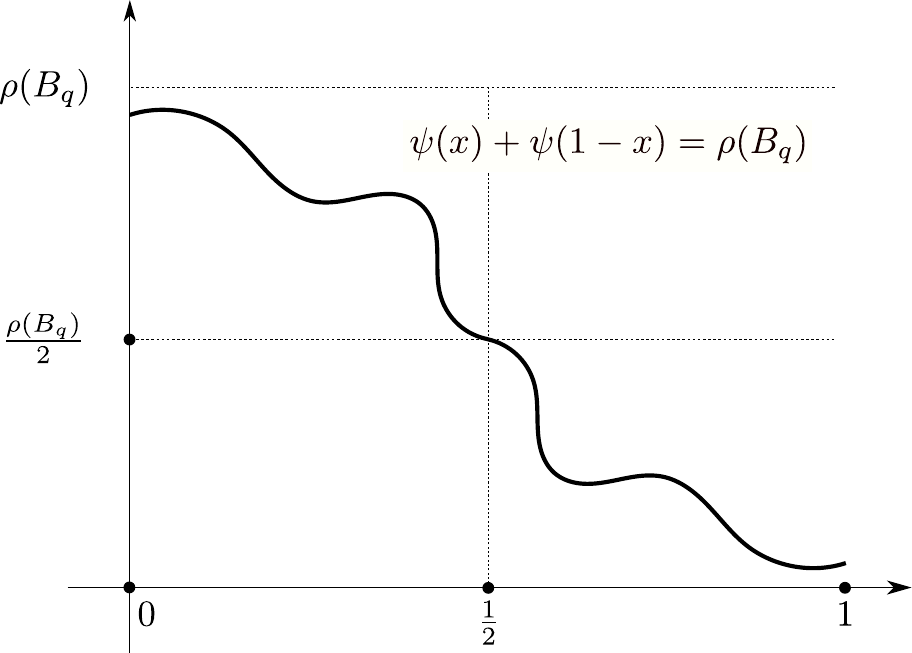}
\caption[Graph of $\psi$]{The graph of a continuous potential $\psi:[0,1]\to\mathbb{R}$ 
satisfying the condition $\psi(x) +\psi(1-x)\equiv \rho(B_{q})$ }
\label{fig:graph1}
\end{figure}

For further discussion on this matter see
Example 3 in \cite{MR1078079}.

\end{example}

\section{Markov Operator and its Eigenmeasures}\label{markov op proper}

If  $\mathcal{R}_{  q}=(X,   \tau,   q)$ is a normalized IFSw, then
$\mathcal{L}_{  q}$ maps $\mathscr{M}_1(X)$ to itself.
Since $\mathscr{M}_1(X)$ is a convex and compact Hausdorff space
we can apply the Tychonoff-Schauder fixed point theorem to ensure the existence
of at least one fixed point for $\mathcal{L}_{  q}$.
In \cite{MR971099} (also \cite{MR625600} or \cite{MR977274})
is shown, under suitable contraction hypothesis,
that $\mathcal{L}_{  q}$ has a unique fixed point $\mu$.
Such probability measure $\mu$ is called the Barnsley-Hutchinson measure
for $\mathcal{R}_{  q}$.

The aim of this section is to present a generalization
of the above result for a non-normalized IFSw.  The central
result is the Theorem \ref{prop-exist-auto-medida}.

\begin{theorem}\label{prop-exist-auto-medida}
Let $\mathcal{R}_{  q}=(X,   \tau,   q)$ be a continuous IFSw
with positive weights.
Then there exists a positive number $\rho\leq \rho(B_{ {q}})$ so that the set
$
\mathcal{G}^{*}( {q})
=
\{
\nu \in \mathscr{M}_1(X): \mathcal{L}_{  q}\nu =\rho\nu
\}
$
is not empty.
\end{theorem}

\begin{proof}
Notice that the mapping
\[
\mathscr{M}_1(X)\ni
\gamma
\mapsto
\frac{\mathcal{L}_{  q}(\gamma) } {\mathcal{L}_{  q}(\gamma)(X)}
\]
sends $\mathscr{M}_1(X)$ to itself. From its convexity
and compactness, in the weak topology which is Hausdorff when $X$ is metric and compact,
it follows from the continuity of $\mathcal{L}_{  q}$ and the Tychonov-Schauder Theorem
that there is at least one probability measure $\nu$ satisfying
$\mathcal{L}_{  q}(\nu)=(\mathcal{L}_{  q}(\nu)(X))\, \nu$.

We claim that 
\begin{align}\label{des1-exi-auto-medida}
n\min_{j\in\{1,\ldots,n\}}\inf_{x\in X} q_j(x)
\leq
\mathcal{L}_{  q}(\gamma)(X)
\leq
n\max_{j\in\{1,\ldots,n\}} \{\|q_j\|_{\infty}\}
\end{align}
for every $\gamma\in \mathscr{M}_1(X)$.
Indeed, 
\begin{align*}
n\min_{j\in\{1,\ldots,n\}}\inf_{x\in X} q_j(x)
\leq
\int_{X} B_{ {q}}(1) \, \text{d}\gamma
=
\int_{X} 1 \, \text{d}[\mathcal{L}_{  q}\gamma]
=
\mathcal{L}_{  q}(\gamma)(X)
\end{align*}
and similarly
\begin{align*}
n\max_{j\in\{1,\ldots,n\}} \{\|q_j\|_{\infty}\}
\geq
\int_{X} B_{ {q}}(1) \, \text{d}\gamma
=
\int_{X} 1 \, \text{d}[\mathcal{L}_{  q}\gamma]
=
\mathcal{L}_{  q}(\gamma)(X).
\end{align*}

From the inequality \eqref{des1-exi-auto-medida} follows that
\[
0
<
\rho
\equiv
\sup\{ \mathcal{L}_{  q}(\nu)(X):
		\mathcal{L}_{  q}(\nu)
		=
		(\mathcal{L}_{  q}(\nu)(X))\, \nu
	\}
<
+\infty.
\]
By a compactness argument one can show the existence of
$\nu\in \mathscr{M}_{1}(X)$ so that
$\mathcal{L}_{  q}\nu=\rho\nu$.
Indeed, let $(\nu_n)_{n\in\mathbb{N}}$
be a sequence such that
$\mathcal{L}_{  q}(\nu_n)(X)\uparrow \rho$,
when $n$ goes to infinity.
Since $\mathscr{M}_1(X)$ is compact metric space
in the weak topology we can assume, up to subsequence,
that $\nu_n\rightharpoonup \nu$. This convergence
together with the continuity of $\mathcal{L}_{  q}$ provides
\[
\mathcal{L}_{  q}\nu
=
\lim_{n\to\infty}\mathcal{L}_{  q}\nu_n
=
\lim_{n\to\infty}\mathcal{L}_{  q}(\nu_n)(X)\nu_n
=
\rho\, \nu,
\]
thus showing that the set
$
\mathcal{G}^{*}( {q})
\equiv\{
\nu \in \mathscr{M}_1(X):
\mathcal{L}_{  q}\nu =\rho\, \nu
\}
\neq
\emptyset
$.

To finish the proof we observe that by using  any
$\nu\in \mathcal{G}^{*}( {q})$, we get the following inequality
\begin{align*}
\rho^N
=
\int_{X} B_{ {q}}^{N}(1)\, \text{d}\nu
\leq
\|B_{ {q}}^{N}\|.
\end{align*}
From this inequality and Gelfand's Formula
follows that $\rho\leq \rho(B_{ {q}})$.
\end{proof}

\begin{theorem}\label{Ruelle IFS cont}
Let $\mathcal{R}_{  q}=(X,   \tau,   q)$ be a continuous IFSw
and suppose $B_{  q}$ has a positive continuous eigenfunction.
Then the eingenvalue $\rho$ of the Markov operator,
provided by Theorem \ref{prop-exist-auto-medida},
satisfies $\rho = \rho_{  q}$.
\end{theorem}

\begin{proof}
Note that we can apply Lemma \ref{powers description}
to ensure that
$N^{-1}\ln\left(B_{  q}^N(1) (x) \right)\to \ln \rho(B_{ {q}})$,
uniformly in $x$ when $N\to\infty$.
By using this convergence, the Lebesgue Dominated Convergence Theorem
and the Jensen Inequality we get
\begin{align*}
\log\rho(B_{ {q}})
=
\lim_{N\to\infty} \frac{1}{N} \int_{X} \ln B_{ {q}}^{N}(1)\, \text{d}\nu
\leq
\lim_{n\to\infty} \frac{1}{N} \ln \int_{X} B_{ {q}}^{N}(1)\, \text{d}\nu
=
\log \rho,
\end{align*}
where $\nu\in \mathcal{G}^{*}( {q})$.
Since $\rho$ is always less than or equal to $\rho(B_{ {q}})$
follows from the above inequality that $\rho=\rho(B_{ {q}})$.
\end{proof}

\section{Holonomic Measure and Disintegrations} \label{holon measures}
An invariant measure for a classical dynamical system $T: X \to X$ on a compact space
is a measure $\mu$ satisfying  for all $f\in C(X,\mathbb{R})$
\[
\int_{X} f(T(x)) d\mu=
\int_{X} f(x) d\mu,
\quad \text{equivalently}\quad
\int_{X} f(T(x))-f(x) d\mu= 0.
\]
From the Ergodic Theory point of view the natural generalization
of this concept for an IFS $\mathcal{R}=(X,   \tau)$ is the
concept of holonomy.

Consider the cartesian product space $\Omega= X \times \{0, ..., n-1\}$
and for each $f\in C(X,\mathbb{R})$ its ``discrete differential''
$df: \Omega \to \mathbb{R}$ defined by $[d_x f](i)\equiv f(\tau_i (x)) -f(x)$.

\begin{definition}\label{invariance}
A measure $\hat{\mu}$ over $\Omega$ is said holonomic,
with respect to an IFS $\mathcal{R}$ if
for all $f\in C(X,\mathbb{R})$ we have
\begin{align*}
\int_{\Omega}[d_x f](i) \, d\hat{\mu}(x,i)=0.
\end{align*}
Notation,
$
\displaystyle \mathcal{H}(\mathcal{R})
\equiv
\{\hat{\mu}  \, | \, \hat{\mu} 
\text{ is a holonomic probability measure with respect to}\ \mathcal{R}\}.
$
\end{definition}

Since $X$ is compact the set of all holonomic probability measures
is obviously convex and compact. It is also not empty because $X$
is compact and any average
\[
\hat{\mu}_{N}
\equiv
\frac{1}{N} \sum_{j=0}^{N-1} \delta_{(x_j, i_j)},
\]
where $x_{j+1} = \tau_{i_j} (x_j)$ and $x_0\in X$ is fixed,
will have their cluster
points in  $\mathcal{H}(\mathcal{R})$. Indeed, for all $N\geq 1$ we have
the following identity
\begin{align*}
\int_{\Omega} [d_x f](i) \, d\hat{\mu}_{N}(x,i)
&=
\frac{1}{N} \sum_{j=0}^{N-1} [d_{x_j} f](i_j)
=
\frac{1}{N} (f(\tau_{i_{N-1}}(x_{N-1}))-f(x_0) ).
\end{align*}
From the above expression is easy to see that
if $\hat{\mu}$ is a cluster point of the sequence $(\hat{\mu}_N)_{N\geq 1}$,
then there is a subsequence $(N_k)_{k\to\infty}$ such that
\begin{align*}
\int_{\Omega} [d_x f](i) \, d\hat{\mu}(x,i)
&=
\lim_{k\to\infty}
\int_{\Omega} [d_x f](i) \, d\hat{\mu}_{N_k}(x,i)
\\
&=
\lim_{k\to\infty}\frac{1}{N_k} (f(\tau_{i_{N_k-1}}(x_{N_k-1}))-f(x_0) )
=0.
\end{align*}

\begin{theorem}[Disintegration]
\label{teo-disintegracao}
Let $X$ and $Y$ be compact metric spaces,
$\hat{\mu}:\mathscr{B}(Y)\to [0,1]$ a Borel probability measure, $T:Y \to X$ a
Borel mensurable function and for each $A\in\mathscr{B}(X)$ define a
probability measure $\mu(A)\equiv \hat{\mu}(T^{-1}(A))$. Then there exists a
family of Borel probability measures $(\mu_{x})_{x \in X}$ on $Y$,
uniquely determined $\mu$-a.e, such that
\begin{enumerate}
\item
$
\mu_{x}(Y\backslash T^{-1}(x))
=
0
$, $\mu$-a.e;

\item
$
\displaystyle
\int_{Y} f\, d\hat{\mu}
= \int_{X}\left[\int_{T^{-1}(x)} \!\!\! f(y)\  d\mu_{x}(y)\right] d\mu(x)
$.
\end{enumerate}
This decomposition is called the \textit{disintegration}
of $\hat{\mu}$, with respect to $T$.
\end{theorem}
\begin{proof}
For a proof of this theorem, see \cite{MR521810} p.78 or \cite{MR2401600}, Theorem 5.3.1.
\end{proof}

In this paper we are interested in disintegrations in cases where
$Y$ is the cartesian product $\Omega\equiv X\times\{0,\ldots,n-1\}$ and $T:\Omega\to X$
is the projection on the first coordinate.
In such cases if $\hat{\mu}$ is any Borel probability measure on $\Omega$,
then follows from the first conclusion
of Theorem \ref{teo-disintegracao} that the disintegration
of $\hat{\mu}$ provides for each $x\in X$ a unique probability measure
$\mu_{x}$ ($\mu$-a.e.) supported on the finite
set $\{(x,0),\ldots,(x,n-1)\}$.
So we can write the disintegration of $\hat{\mu}$
as $d\hat{\mu}(x,i)= d\mu_{x}(i)d\mu(x)$, where here we are abusing notation
identifying $\mu_x(\{(x,j)\})$ with $\mu_x(\{j\})$.

Now  we take
$\hat{\mu} \in \mathcal{H}(\mathcal{R})$ and $f:\Omega\to \mathbb{R}$
as being any bounded continuous function, depending only on its first coordinate.
From the very definition of holonomic measures we have the following equations
\[
\int_{\Omega}[d_x f](i) \, d\hat{\mu}(x,i)  =  0
\Longleftrightarrow
\int_{\Omega}f(\tau_i (x))  \, d\hat{\mu}(x,i)  = \int_{\Omega} f(x) \, d\hat{\mu}(x,i)
\]
by disintegrating both sides of the second equality above we get that
\[
\int_{X}\int_{\{0, ..., n-1\}} f(\tau_i (x))  \, d\mu_{x}(i)d\mu(x)
=
\int_{X}\int_{\{0, ..., n-1\}} f(x) \, d\mu_{x}(i)d\mu(x).
\]
Recalling that $\mu_x$ is a probability measure
follows from the above equation that
\[
\int_{X}\sum_{i=0}^{n-1} \mu_{x}(i) f(\tau_i (x))  \, d\mu(x)
=
\int_{X} f(x) \,d\mu(x).
\]

The last equation establish a natural link
between holonomic measures for an IFS $\mathcal{R}$
and disintegrations. Given an IFS $\mathcal{R}=(X,\tau)$
and $\hat{\mu}\in \mathcal{H}(\mathcal{R})$ we can use the previous
equation to define an IFSpdp $\mathcal{R}_{ {q}}=(X,\tau, {q})$, where the
weights $q_i(x)=\mu_x(\{i\})$. If $B_{ {q}}$ denotes the transfer operator
associated to $\mathcal{R}_{ {q}}$ we have from the last equation the following
identity
\[
\int_{X} B_{ {q}}(f)  \, d\mu(x)
  =
  \int_{X} f(x) \,d\mu(x).
\]
Since in the last equation $f$ is an arbitrary bounded measurable
function, depending only on the first coordinate, follows that
the Markov operator associated to the IFSpdp $\mathcal{R}_{ {q}}$ satisfies
\[
\mathcal{L}_{ {q}} (\mu) = \mu.
\]
In other words the ``second marginal'' $\mu$ of a holonomic measure $\hat{\mu}$
is always an eingemeasure for the Markov operator associated to the
IFSpdp $\mathcal{R}_{ {q}}=(X,\tau, {q})$ defined above.

Reciprocally. Since the last five equations are equivalent,
given an IFSpdp $\mathcal{R}_{ {q}}=(X,\tau, {q})$
such that the associated Markov operator has at least one fix point, i.e.,
$\mathcal{L}_{  q} (\mu) =\mu$, then it is possible to define a
holonomic probability measure $\hat{\mu}\in \mathcal{H}(\mathcal{R})$
given by
$d\hat{\mu}(x,i)= d\mu_x (i) \,d\mu(x)$,
where
$\mu_x (i)\equiv q_{i}(x)$.
This Borel probablity measure on $\Omega$ will be called
the {\bf holonomic lifting} of $\mu$, with respect to $\mathcal{R}_{ {q}}$.

\section{Entropy and Pressure for IFSw}\label{entropy and pressure}

In this section we introduce the notions of topological pressure and entropy.
We adopted variational formulations for both concepts because it allow us treat
very general IFSw. These definitions introduced here are inspired, and generalizes,
the recent theory of such objects in the context of symbolic dynamics
of the left shift mapping acting
on $X= M^{\mathbb{N}}$, where $M$ is an uncoutable compact metric space.

As in the previous section the mapping $T:\Omega\to X$ denotes the projection
on the first coordinate. Even when not explicitly mentioned,
any disintegrations of a probability measure $\hat{\nu}$, defined over $\Omega$,
will be from now considered with respect to $T$.

We write $B_1$ to denote the transfer operator $B_q$, 
where the weights $q_i(x)\equiv 1$ for all $x\in X$ and $i=0,\ldots,n-1$.

\begin{definition}[Average and Variational Entropies]\label{entropy}
Let $\mathcal{R}$ be an IFS, $\hat{\nu} \in \mathcal{H}(\mathcal{R})$ and
$d\hat{\nu}(x,i)=d\nu_{x}(i)d\nu(x)$ a disintegration of $\hat{\nu}$, with
respect to $T$.
The variational and average entropies of $\hat{\nu}$ are defined, respectively, by
\[
h_v(\hat{\nu})
\equiv
\inf_{ \substack{ g\in C(X, \mathbb{R}) \\ g>0  } }
\left\{ \int_X \ln \frac{B_{1}(g)(x)}{ g(x) }  d\nu(x) \right\}
\]
and
\[
h_a(\hat{\nu})
\equiv
-\int_X \sum_{i=0}^{n-1} q_{i}(x)  \ln q_{i}(x)\,  d\nu(x) ,
\]
where $q_i(x) \equiv \nu_{x}(i)$, for all $x\in X$ and $i=0,\ldots,n-1$.
\end{definition}

\begin{definition}[Optimal Function]\label{Optimal Variational g}
Let $\mathcal{R}$ be an IFS, $\hat{\nu} \in \mathcal{H}(\mathcal{R})$,
$d\hat{\nu}(x,i)=d\nu_{x}(i)d\nu(x)$ a disintegration of $\hat{\nu}$, with
respect to $T$ and $q_i(x) \equiv \nu_{x}(i)$, for all $x\in X$ and $i=0,\ldots,n-1$.
We say that a positive function $g \in C(X, \mathbb{R})$ is optimal,
with respect to the the IFSpdp $\mathcal{R}_{  q}=(X,   \tau,   q)$
if for all $i=0, ...,n-1$ we have
\[
q_i(x)
=
\frac{g(\tau_i(x))}
{B_1(g)(x)}.
\]
\end{definition}

As the reader probably already noted,
the expression of the average entropy $h_{a}$ is a familiar one.
In the sequel we prove a theorem establishing some relations
between the two previous defined concepts of entropies.
This result is a useful tool when doing
some computations regarding the pressure functional, which will be
defined later. Before state the theorem we recall a fundamental
inequality regarding non-negativity of the conditional entropy
of two probability vectors. See, for example,
the reference \cite{MR1085356} for a proof.

\begin{lemma}\label{entropy inequality ln}
For any probability vectors $(a_0,\ldots,a_{n-1})$ and $(b_0,\ldots,b_{n-1})$
we have
\[
-\sum_{i=0}^{n-1}  a_{i} \ln( a_{i})
\leq
- \sum_{i=0}^{n-1} a_{i} \ln( b_{i})
\]
and the equality is attained iff $a_{i}=b_{i}$.
\end{lemma}

\begin{theorem}\label{entropy inequality ha hv}
Let $\mathcal{R}$ be an IFS, $\hat{\nu} \in \mathcal{H}(\mathcal{R})$,
$d\hat{\nu}(x,i)=d\nu_{x}(i)d\nu(x)$ a disintegration of $\hat{\nu}$, with
respect to $T$ and $\mathcal{R}_{ {q}}=(X, {\tau}, {q})$ the IFSw with $q_i(x)=\nu_{x}(i)$
for all $x\in X$ and $i=0,\ldots,n-1$. Then
\begin{enumerate}
\item $0\leq h_a(\hat{\nu}) \leq h_v(\hat{\nu}) \leq \ln n$;
\item if there exists some optimal function $g^*$, with respect to $\mathcal{R}_{ {q}}$,
then
\[
h_a(\hat{\nu})
=
h_v(\hat{\nu})
=
\int_{X} \ln \frac{B_{1}(g^*)}{ g^* }  d\nu,
\]

\end{enumerate}

\end{theorem}
\begin{proof}
	We first prove item $\mathit{1}$. Since $0 \leq q_{i}(x)\equiv \nu_{x}(i) \leq 1$
	follows from the definition of average entropy that $h_a(\hat{\nu}) \geq 0$.
	
	From the definition of variational entropy we obtain
	\[
	h_v(\hat{\nu})
	=
	\inf_{ \substack{ g\in C(X, \mathbb{R}) \\ g>0  } }
	\left\{ \int_{X} \ln \frac{B_{1}(g)}{ g }  d\nu \right\}
	\leq
	\int_{X} \ln \frac{B_{1}(1)}{1}  d\nu= \ln n.
	\]

To finish the proof of item 1 remains to show that $h_a(\hat{\nu}) \leq h_v(\hat{\nu})$.
Let $g:X\to\mathbb{R}$ be continuous positive function and
define for each $x\in X$ a probability vector $(p_0(x),\ldots,p_{n-1}(x))$,
where $p_j(x)=g(\tau_j(x))/B_{1}(g)(x)$, for each $j=0,\ldots,n-1$.
From Lemma  \ref{entropy inequality ln}
and the properties of the holonomic measures
we get the following inequalities for any continuous and positive function $g$
	\begin{align}\label{desigualdade-ha-otimal}
	h_a(\hat{\nu})
	&=
	- \int_{X} \sum_{j=0}^{n-1} q_j(x) \ln( q_j(x)) d\nu
	\leq
	- \int_{X} \sum_{j=0}^{n-1} q_j \ln\left( \frac{g\circ\tau_j}{B_{1}(g)}\right) d\nu
	\nonumber
	\\[0.2cm]
	&=
	- \int_{X}
		\left[
			\sum_{j=0}^{n-1} q_j \ln (g\circ\tau_j) -
			\sum_{j=0}^{n-1} q_j \ln (B_{1}(g))
		\right]
	  d\nu
	\nonumber
	\\[0.2cm]
	&=
	- \int_{X} B_{ {q}}(\ln g) \,  d\nu +
	\int_{X} \ln (B_{1}(g)) d\nu
	\nonumber
	\\
	&=
	- \int_{X}  \ln g\,  d\nu + \int_{X}\ln (B_{1}(g))\, d\nu
	\nonumber
	\\[0.3cm]
	&=
	\int_{X} \ln \frac{B_{1}(g)}{g}\, d\nu,
	\end{align}
	Therefore
	\[
	h_a(\hat{\nu})
	\leq
	\inf_{ \substack{ g\in C(X, \mathbb{R}) \\ g>0  } }
	\left\{ \int_{X} \ln \frac{B_{1}(g)}{g}\,  d\nu \right\}
	=
	h_v(\hat{\nu})
	\]
	and the item 1 is proved.
	
	Proof of  item $\mathit{2.}$
	From Lemma \ref{entropy inequality ln}
	it follows that the equality in \eqref{desigualdade-ha-otimal} is attained
	for any optimal function with respect to $\mathcal{R}_{ {q}}$.
	Since we are assuming the existence of at least one optimal function $g^*$
	we have
	\[
	h_a(\hat{\nu}) =  \int_{X} \ln \frac{B_{1}(g^*)}{ g^* }  d\nu
	\geq
	h_v(\hat{\nu}).
	\]
	Since the reverse inequality is always valid we are done.
\end{proof}

We now introduce the natural generalization
of the concept of topological pressure of a continuous
potential.

\begin{definition}\label{Pressure}
Let $\psi:X \to \mathbb{R}$ be a positive continuous function and
$\mathcal{R}_{  \psi}\equiv (X,   \tau,  {\psi\!*\!\tau} )$ an IFSw,
where the weights are give by $(\psi\!*\!\tau)_{i}(x)\equiv \psi(\tau_i(x))$.
The topological pressure  of $\psi$, with respect to the IFSw
$\mathcal{R}_{  \psi}$, is defined by the following expression
\[
P(\psi)
\equiv
\sup_{\hat{\nu} \in \mathcal{H}(\mathcal{R})}
\
\inf_{ \substack{ g\in C(X, \mathbb{R}) \\ g>0  } }
\left\{ \int_{X} \ln \frac{B_{  {\psi*\tau}   }(g)}{ g }  d\nu \right\},
\]
where $d\nu_{x}(i)d\nu(x)=d\hat{\nu}(x,i)$ is a disintegration of $\hat{\nu}$, with
respect to $\mathcal{R}$.
\end{definition}

\begin{lemma}\label{basic pressure inequality}
Let $\psi:X \to \mathbb{R}$ be a positive continuous function and
$\mathcal{R}_{  \psi}\equiv (X,   \tau,  {\psi\!*\!\tau} )$ the IFSw
above defined. Then the topological pressure of $\psi$ is alternatively given by
\[
P(\psi)=
\sup_{\hat{\nu} \in \mathcal{H}(\mathcal{R})}
\left\{ h_{v}(\hat{\nu})+  \int_{X} \ln \psi\,  d\nu\right\}.
\]
\end{lemma}
\begin{proof} To get this identity we only need to use the pressure's definition and the
	basic properties of the transfer operator as follows
	\begin{align*}
	P(\psi)
	&\equiv
	\sup_{\hat{\nu} \in \mathcal{H}(\mathcal{R})}\
	\inf_{ \substack{ g\in C(X, \mathbb{R}) \\ g>0  } }\
	\left\{ \int_{X} \ln \frac{B_{  {\psi*\tau}}(g)}{ g }\  d\nu \right\}
	\\[0.3cm]
	&=
	\sup_{\hat{\nu} \in \mathcal{H}(\mathcal{R})}\
	\inf_{ \substack{ g\in C(X, \mathbb{R}) \\ g>0  } }\
	\left\{\int_{X} \ln \psi\,  d\nu -   \int_{X} \ln \psi\,  d\nu+
	\int_{X} \ln \frac{B_{  {\psi*\tau}}(g)}{g}\,  d\nu \right\}
	\\[0.3cm]
	&=
	\sup_{\hat{\nu} \in \mathcal{H}(\mathcal{R})}
	\left\{
		\int_{X} \ln \psi \, d\nu +
		\inf_{ \substack{ g\in C(X, \mathbb{R}) \\ g>0  } }
		\int_{X} \ln \frac{B_{  {\psi*\tau}}(g)}{ \psi g } \, d\nu
	\right\}
	\\[0.3cm]
	&=
	\sup_{\hat{\nu} \in \mathcal{H}(\mathcal{R})}
	\left\{
		\int_{X} \ln \psi \, d\nu +
		\inf_{ \substack{ \tilde{g}\in C(X, \mathbb{R}) \\ \tilde{g}>0  } }
		\int_{X} \ln \frac{B_{1}(\tilde{g})}{ \tilde{g} } \, d\nu \right\}
	\text{, where } \psi g = \tilde{g}
	\\[0.3cm]
	&=
	\sup_{\hat{\nu} \in \mathcal{H}(\mathcal{R})}
	\left\{\int_{X} \ln \psi \, d\nu +  h_{v}(\hat{\nu}) \right\}.
	\end{align*}
\end{proof}

\begin{definition}[Equilibrium States]
Let $\mathcal{R}$ be an IFS and $\hat{\mu} \in \mathcal{H}(\mathcal{R})$.
We say that the holonomic measure $\hat{\mu}$ is an equilibrium state for
$\psi$ if
\[
h_{v}(\hat{\mu})+  \int_{X} \ln( \psi(x))\, d\mu(x) = P(\psi).
\]
\end{definition}

\begin{lemma}
Let $X$ and $Y$ compact separable metric spaces and $T:Y\to X$ a continuous mapping.
Then the push-forward mapping
$\Phi_{T}\equiv \Phi:\mathscr{M}_{1}(Y)\to \mathscr{M}_{1}(X)$ given by
\[
\Phi(\hat{\mu})(A)= \hat{\mu}(T^{-1}(A)),
\quad\text{where}\ \hat{\mu}\in \mathscr{M}_{1}(Y)\ \text{and} \ A\in \mathscr{B}(X)
\]
is weak-$*$ to weak-$*$ continuous.
\end{lemma}

\begin{proof}
Since we are assuming that $X$ and $Y$ are separable compact metric spaces then
we can ensure that the weak-$*$ topology of both $\mathscr{M}_{1}(Y)$ and $\mathscr{M}_{1}(X)$
are metrizable. Therefore is enough to prove that $\Phi$ is sequentially continuous.
Let $(\hat{\mu}_n)_{n\in\mathbb{N}}$ be a sequence
in $\mathscr{M}_{1}(Y)$ so that $\hat{\mu}_n\rightharpoonup \hat{\mu}$.
For any continuous real function $f:X\to\mathbb{R}$
we have from change of variables theorem that
\[
\int_{X} f\, d[\Phi(\hat{\mu}_n)]
=
\int_{Y} f\circ T\, d\hat{\mu}_n,
\]
for any $n\in\mathbb{N}$. From the definition
of the weak-$*$ topology follows that the rhs above
converges when $n\to\infty$ and we have
\[
\lim_{n\to\infty}\int_{X} f\, d[\Phi(\hat{\mu}_n)]
=
\lim_{n\to\infty}\int_{Y} f\circ T\, d\hat{\mu}_n
=
\int_{Y} f\circ T\, d\hat{\mu}
=
\int_{X} f\, d[\Phi(\hat{\mu})].
\]
The last equality shows that $\Phi(\hat{\mu}_n)\rightharpoonup \Phi(\hat{\mu})$
and consequently the weak-$*$ to weak-$*$ continuity of $\Phi$.
\end{proof}

For any $\hat{\nu}\in\mathcal{H}(\mathcal{R})$ it is always possible to disintegrate
it as $d\hat{\nu}(x,i)= d\nu_{x}(i)d[\Phi(\hat{\nu})](x)$, where $\Phi(\hat{\nu})\equiv \nu$ is
the probability measure on $\mathscr{B}(X)$, defined for any $A\in\mathscr{B}(X)$ by
\begin{align}\label{def-Phi}
\nu(A)\equiv \Phi(\hat{\nu})(A) \equiv  \hat{\nu}(T^{-1}(A)),
\end{align}
where
$T:\Omega\to X$ is the cannonical projection of the first coordinate.
This observation together with the previous lemma allow us to
define a continuous mapping from $\mathcal{H}(\mathcal{R})$
to $\mathscr{M}_{1}(X)$ given by $\hat{\nu}\longmapsto \Phi(\hat{\nu})\equiv \nu$.

We now prove a theorem ensuring the existence of equilibrium states for
any continuous positive function $\psi$.
Although this theorem has clear and elegant proof and works
in great generality it has the disadvantage of
providing no description of the set of equilibrium states.

\begin{theorem}[Existence of Equilibrium States]
	Let $\mathcal{R}$ be an IFS and $\psi:X\to\mathbb{R}$ a positive continuous function.
	Then the set of equilibrium states for $\psi$ is not empty.
\end{theorem}

\begin{proof}
As we observed above we can define a weak-$*$ to weak-$*$ continuous mapping
\[
\mathcal{H}(\mathcal{R})\ni \hat{\nu}\longmapsto \nu\in \mathscr{M}_{1}(X),
\]
where $d\hat{\nu}(x,i)=d\nu_{x}(i)d\nu(x)$ is the above constructed
disintegration of $\hat{\nu}$.
From this observation follows that
for any fixed positive continuous $g$ we have that the mapping
$
\mathcal{H}(\mathcal{R})\ni\hat{\nu}
\longmapsto
\int_X \ln (B_{1}(g)/ g) \, d\nu
$
is continuous with respect to the weak-$*$ topology. Therefore
the mapping
\[
\mathcal{H}(\mathcal{R})\ni\hat{\nu}
\longmapsto
\inf_{ \substack{ g\in C(X, \mathbb{R}) \\ g>0  } }
\left\{ \int_{X} \ln \frac{B_{1}(g)}{g}\,  d\nu \right\}
\equiv
h_v(\hat{\nu}).
\]
is upper semi-continuous (USC) which implies by standard results that
the following mapping is also USC
\[
\mathcal{H}(\mathcal{R})\ni\hat{\nu}
\longmapsto
h_{v}(\hat{\nu})+  \int_{X} \ln( \psi(x))\, d\nu(x).
\]
Since $\mathcal{H}(\mathcal{R})$ is compact in the weak-$*$ topology and
the above mapping is USC then follows from Bauer maximum principle that
this mapping attains its supremum at some $\hat{\mu}\in \mathcal{H}(\mathcal{R})$, i.e.,
\[
\sup_{\hat{\nu} \in \mathcal{H}(\mathcal{R})}
\left\{\int_{X} \ln \psi \, d\nu +  h_{v}(\hat{\nu}) \right\}
=
\int_{X} \ln \psi \, d\mu +  h_{v}(\hat{\mu})
\]
thus proving the existence of at least one equilibrium state.
\end{proof}

\subsection{Pressure Differentiability and Equilibrium States}

In this section we consider the functional $p:C(X,\mathbb{R})\to \mathbb{R}$
given by
\begin{align}\label{def-funcional-p}
p(\varphi) = P(\exp(\varphi)).
\end{align}
It is immediate to verify
that $p$ is convex and finite valued functional.
We say that a Borel signed measure $\nu\in\mathscr{M}_{s}(X)$ is a
{\bf subgradient} of $p$ at $\varphi$ if it satisfies the following
subgradient inequality $p(\eta)\geq p(\varphi)+\nu(\eta-\varphi)$.
The set of all subgradients at $\varphi$ is called {\bf subdifferential} of $p$
at $\varphi$ and denoted by $\partial p(\varphi)$.
It is well-known that if $p$ is a continuous mapping
then $\partial p(\varphi)\neq \emptyset$ for any $\varphi\in C(X,\mathbb{R})$.

We observe that for any pair $\varphi,\eta\in C(X,\mathbb{R})$
and $0<t<s$, follows from the convexity of $p$ the following inequality
$
s( p(\varphi+t\eta)-p(\varphi))\leq t(p(\varphi+s\eta)-p(\varphi)).
$
In particular, the one-sided directional derivative
$d^{+}p(\varphi):C(X,\mathbb{R})\to \mathbb{R}$ given by
\[
d^{+}p(\varphi)(\eta)
=
\lim_{t\downarrow 0} \frac{p(\varphi+t\eta)-p(\varphi)}{t}
\]
is well-defined for any $\varphi\in C(X,\mathbb{R})$.

\begin{theorem} \label{teo-gateaux-unicidade}
For any fixed $\varphi\in C(X,\mathbb{R})$ we have
\begin{enumerate}
\item
the signed measure $\nu\in\partial p(\varphi)$ iff
$\nu(\eta)\leq d^{+}p(\varphi)(\eta)$ for all $\eta\in C(X,\mathbb{R})$;

\item
the set $\partial p(\varphi)$ is a singleton iff $d^{+}p(\varphi)$
is the G\^ateaux derivative of $p$ at $\varphi$.
\end{enumerate}
\end{theorem}
\begin{proof}
This theorem is consequence of Theorem 7.16 and Corollary 7.17
of the reference \cite{MR2378491}.
\end{proof}

\begin{theorem}
Let $\mathcal{R}$ be an IFS, $\psi:X\to\mathbb{R}$ a positive continuous function
and $\Phi$ defined as in \eqref{def-Phi}.
If the functional $p$ defined on \eqref{def-funcional-p}
is G\^ateaux differentiable at $\varphi\equiv \log \psi$
then
\[
\# \{\Phi(\hat{\mu}):\ \hat{\mu}\ \text{is an equilibrium state for}\ \psi \} =1.
\]
\end{theorem}

\begin{proof}
Suppose that $\hat{\mu}$ is an equilibrium state for $\psi$. Then
we have from the definition of the pressure that
\begin{align*}
p(\varphi+t\eta) -p(\varphi)
&=
P(\psi\exp(t\eta))-P(\psi)
\\
&\geq
h_{v}(\hat{\mu})+\int_{X} \ln \psi\ d\mu + \int_{X} t\eta\, d\mu
-
h_{v}(\hat{\mu})-\int_{X} \ln \psi\ d\mu
\\
&=
t\int_{X} \eta\, d\mu.
\end{align*}
Since we are assuming that $p$ is G\^ateaux differentiable at $\varphi$
follows from the above inequality that $\mu(\eta)\leq d^{+}p(\varphi)(\eta)$
for all $\eta\in C(X,\mathbb{R})$. From this inequality and
Theorem \ref{teo-gateaux-unicidade}
we can conclude  that $\partial p(\varphi)=\{\mu\}$.
Therefore for all equilibrium state $\hat{\mu}$ for $\psi$
we have $\Phi(\hat{\mu})=\partial p(\varphi)$, thus finishing the proof.
\end{proof}

\section{Applications and Constructive Approach to Equilibrium States} \label{applic eq state}

In this section we show how one can construct equilibrium states
using the spectral analysis of the transfer operator. The results
present here are based on Theorem 3.27 of \cite{CO-DynProg17} which is
a result about Dynamic Programming. In order to make its statement
and this section as self-contained as possible we provided here the needed
background.

The Theorem 3.27 mentioned above is a kind of generalized version
of the classical Ruelle-Perron-Frobenius theorem. To be more precisely,
we consider a sequential decision-making process
$S=\{X, A, \xi, f, u, \delta\}$ derived from
$\mathcal{R}_{  q}=(X,   \tau,   q)$ by choosing
$A=\{0, ..., n-1\}$ the action set, $\xi(x)=A$ for any
$x \in X$, $f(x,i)=\tau_{i}(x)$ the dynamics, $u(x,i)= \ln q_i(x)$ and $\delta$
a discount function so that for some increasing function $\gamma:[0,\infty) \to [0,\infty)$
with $\lim_{n \to \infty} \gamma^n(t)=0$
we have $|\delta(t_2) - \delta(t_1)| \leq \gamma(|t_2 - t_1|)$
for any $t_1, t_2 \in \mathbb{R}$.

In this setting we consider a parametric family of discount functions
$\delta_{n}:[0,+\infty) \to \mathbb{R}$, where $\delta_{n}(t) \to I(t)=t$,
pointwise and the normalized limits
$\lim_{n \to \infty} w_{n}(x) -\max w_{n}$ of the fixed points
\[
w_{n}(x)
=
\ln \sum_{i=0}^{n-1} \exp\Big(  u(x,i)+\delta_{n}(w_{n}(f(x,i)))   \Big)
\]
of a variable discount decision-making process $S_n=\{X, A, \psi, f, u, \delta_{n}\}$
defined by a continuous and bounded immediate reward  $u:X \times A \to \mathbb{R}$
and a sequence of discount functions $(\delta_n)_{n\geq0}$,
satisfying the admissibility conditions:
\begin{enumerate}
  \item[$\mathit{1.}$] the contraction modulus  $\gamma_{n}$ of the variable discount $\delta_{n}$ is also a variable discount function;
  \item[$\mathit{2.}$] $\delta_{n}(0)=0$ and $\delta_{n}(t) \leq t$ for any $t\in [0,\infty)$;
  \item[$\mathit{3.}$]  for any fixed $\alpha >0$ we have $\delta_{n}(t +\alpha) - \delta_{n}(t)\to \alpha$, when $n\to \infty$,
  uniformly in $t>0$.
\end{enumerate}

\begin{theorem}[\cite{CO-DynProg17}]\label{spectral radii}
   Let $\mathcal{R}_{  q}=(X,   \tau,   q)$ be an IFSw,
   such that the above defined immediate associated return $u$ satisfy:
   \begin{enumerate}
   	\item $u$ is uniformly $\delta$-bounded;
   	\item $u$ is uniformly $\delta$-dominated.
   \end{enumerate}
   Then there exists a positive and continuous eigenfunction
   $h_{ {q}}$ such that
   $B_{ {q}}(h_{ {q}})=\rho(B_{ {q}}) h_{ {q}}$.
\end{theorem}

As pointed out in \cite{CO-DynProg17},
the hypothesis of this theorem are not so restrictive as it initially looks like.
In fact, a lot of variable discount parametric families satisfies our requirements and,
the uniformly $\delta$-bounded and $\delta$-dominated property are satisfied
for $u$ in the Lipschitz and H\"older spaces provided that the IFS is contractive.
So the hypothesis placed on the variable discount allow us to apply
this theorem for a large class of weights.

\begin{corollary} \label{normalization}
Under the same hypothesis of Theorem~\ref{spectral radii},
given an IFSw $\mathcal{R}_{  q}=(X,   \tau,   q)$
there exists a continuous IFSpdp
$\mathcal{R}_{  p}=(X,   \tau,   p)$
called the normalization of $\mathcal{R}_{  q}$, where
\[
p_j (x)
=
q_j(x)
\frac{ h_{ {q}}(\tau_j(x))}{\rho(B_{ {q}}) \,h_{ {q}}(x)} ,
\; j=0, ..., n-1.
\]
\end{corollary}

\begin{theorem}[Variational principle]\label{VarPrinciple}
Let $\psi:X \to \mathbb{R}$ be a positive and continuous function
and $\mathcal{R}_{ {\psi*\tau}}=(X,   \tau,  {\psi*\tau})$ an IFSw.
Assume that $\mathcal{R}_{ {\psi*\tau}}$ satisfies the hypothesis of
Theorem \ref{spectral radii}.
Consider $\mathcal{R}_{  p}=(X,   \tau,   p)$
be the normalization of $\mathcal{R}_{ {\psi*\tau}}$
given by Corollary~\ref{normalization} and
$\hat{\mu}=\Phi(\mu) \in \mathcal{H}(\mathcal{R})$
be the holonomic lifting
of the fixed probability measure $\mu$ of $\mathcal{L}_{ {p}}$.
Then $\hat{\mu}$, is a equilibrium state for $\psi$ and
\begin{align}\label{2formulas-pressao}
P(\psi)
=
\ln \rho(B_{ {\psi*\tau}})
=
\lim_{N \to \infty} \frac{1}{N} \ln\left(B_{ {\psi*\tau}}^N(1) (x) \right),
\end{align}
where $\rho(B_{ {\psi*\tau}})$ is the spectral radius of the transfer
operator $B_{ {\psi*\tau}}$ acting on $C(X,\mathbb{R})$ and
the convergence in the above limit is uniform and independent of the
choice of $x\in X$.
\end{theorem}
\medskip

\begin{proof}
	From item $\mathit{2}$ of Theorem~\ref{entropy inequality ha hv}
	follows that $h_a(\hat{\nu}) \leq h_v(\hat{\nu})$.
	From Lemma \ref{entropy inequality ln} and Corollary \ref{normalization}
	we get the following inequality
	\begin{align*}
	h_a(\hat{\nu})
	&=
	- \int_{X} \sum_{j=0}^{n-1} (\psi\circ\tau_j) \ln q_j\, d\nu
	\leq
	- \int_{X} \sum_{j=0}^{n-1} (\psi\circ\tau_j) \ln p_j\, d\nu
	\\
	&=
	\ln \rho(B_{ {\psi*\tau}}) - \int_{X} \ln \psi \, d\nu,
	\end{align*}
	with equality being attained if $\hat{\nu}=\hat{\mu}$, where
	$\hat{\mu}$ is the holonomic lifting of the fixed point $\mathcal{L}_{ {p}}(\mu)=\mu$.
	So we have
	\[
	h_a(\hat{\mu}) =h_v(\hat{\mu})= \ln \rho(B_{ {\psi*\tau}}) - \int_{X} \ln \psi  d\mu.
	\]
	From Lemma~\ref{basic pressure inequality} we have
	$
	\sup\{ 	h_{v}(\hat{\nu})+  \int_{X} \ln \psi\,  d\nu:
			\hat{\nu} \in \mathcal{H}(\mathcal{R})
		\}
	=
	P(\psi)
	$, so for every $\hat{\nu} \in \mathcal{H}(\mathcal{R})$ we have
	\[
	P(\psi) \geq h_{v}(\hat{\nu}) +  \int_{X} \ln \psi\,  d\nu
	\geq
	h_{a}(\hat{\nu})+ \int_{X} \ln \psi\,  d\nu.
	\]
	On the other hand, the identity
	$h_a(\hat{\mu}) =h_v(\hat{\mu})=  \ln \rho(B_{ {\psi*\tau}}) -\int_{X} \ln \psi\,  d\mu$
	is equivalent to both
	$h_{a}(\hat{\mu})+  \int_{X} \ln \psi\,  d\mu = \ln \rho(B_{ {\psi*\tau}})$
	and
	$h_{v}(\hat{\mu})+ \int_{X} \ln \psi  d\mu = \ln \rho(B_{ {\psi*\tau}})$.
	Thus showing that $P(\psi)\geq \ln \rho(B_{ {\psi*\tau}})$.
	Recall that the pressure is defined by
	\[
	P(\psi)
	\equiv
	\sup_{\hat{\nu} \in \mathcal{H}(\mathcal{R})}
	\
	\inf_{ \substack{ g\in C(X, \mathbb{R}) \\ g>0  } }
	\left\{ \int_{X} \ln \frac{B_{  {\psi*\tau}   }(g)}{ g }  d\nu \right\}.
	\]
    The eigenfunction $h_{ {\psi*\tau}}$ is positive and continuous
    so we have for any fixed $\hat{\nu} \in \mathcal{H}(\mathcal{R})$
    the following upper bound
    \[
    \inf_{ \substack{ g\in C(X, \mathbb{R}) \\ g>0  } }
	\left\{ \int_{X} \ln \frac{B_{  {\psi*\tau} }(g)}{ g }  d\nu \right\}
	\leq
	\int_{X} \ln \frac{B_{  {\psi*\tau} }(h_{ {\psi*\tau}})}{h_{ {\psi*\tau}} }  d\nu
	=\ln \rho(B_{ {\psi*\tau}}).
	\]
    By taking the supremum over $\mathcal{H}(\mathcal{R})$, on both sides of the last
    inequality, we get that
    $P(\psi)\leq \ln \rho(B_{ {\psi*\tau}})$. Since we already shown the reverse inequality
    the first equality in \eqref{2formulas-pressao} is proved. The second equality claimed in
    the theorem statement is now a straightforward application
    of Lemma \ref{powers description}.
\end{proof}

\section*{Acknowledgments}
We would like to express our thanks to the organizers of the IV EBED, 
where this work began. We also thank Krzysztof Le\'sniak and Edgar Matias 
for their helpful comments on the early version of this manuscript.  
Leandro Cioletti thanks CNPq for the financial support.

\end{document}